\documentclass[12 pt]{amsart}

\usepackage{xypic}
\usepackage{amssymb}
\usepackage{dsfont}
\usepackage{hyperref}

\usepackage{pythonhighlight}

\usepackage{comment}
\usepackage{amsmath}
\usepackage{mathtools}
\usepackage{enumitem}

\usepackage{tikz-cd}

\xyoption{all}

\numberwithin{equation}{section}
\numberwithin{figure}{section}
\theoremstyle{plain}
\newtheorem{thm}{Theorem}[section]
\newtheorem*{thm*}{Theorem}

\newtheorem{prop}[thm]{Proposition}
\newtheorem{cor}[thm]{Corollary}
\newtheorem{fact}[thm]{Fact}
\theoremstyle{definition}

\newtheorem{defn}[thm]{Definition}

\newtheorem{lem}[thm]{Lemma}

\newtheorem{ques}[thm]{Question}

\theoremstyle{remark}

\newtheorem*{acknowledgement*}{Acknowledgement}

\newcommand\suchthat{\;\ifnum\currentgrouptype=16 \middle\fi|\;}

\def\F{\mathcal{F}}
\def\d{\delta}

\newcommand{\Z}{\mathbb{Z}}

\def\cyc{\operatorname{cyc}}
\newcommand{\eps}{\varepsilon}
\newcommand{\llangle}{\langle\! \langle}
\newcommand{\rrangle}{\rangle\! \rangle}

\title{Random semigroups \'a la Gromov}
\author{Guy Blachar}
\address{Universit\'{e} Paris-Dauphine -- CEREMADE, Place de Lattre de Tassigny, Paris, France}
\email{guy.blachar@gmail.com}
\author{Be'eri Greenfeld}
\address{Department of Mathematics and Statistics, CUNY Hunter College, 695 Park Avenue, New York NY 10065, USA} \email{beeri.greenfeld@hunter.cuny.edu}

\author{Margaret Shideler}
\address{Department of Mathematics and Statistics, CUNY Hunter College, 695 Park Avenue, New York NY 10065, USA} \email{margaret.shideler@hunter.cuny.edu}


\begin{document}

\begin{abstract}
We introduce a model of random semigroups analogous to Gromov's density model for random groups. For a fixed number $n\ge 2$ of generators and a density parameter $\delta\in (0,1)$, we impose a uniformly chosen set of $\lfloor n^{2\delta l}\rfloor$ relations of the form $u=v$, where $u,v$ are distinct words of length at most $l$, and study the resulting semigroup as $l\to\infty$. We prove that, with overwhelming probability, above density $1/2$ the semigroup is trivial, while below density $1/2$ it contains a free subsemigroup of rank two, has a Golod-Shafarevich semigroup algebra over every field, and has a homomorphic image which generates a non-elementary hyperbolic group. We also show that the small cancellation condition $C(m)$ holds with overwhelming probability whenever $\delta<1/2m$. Finally, for $n\ge 3$, at critical density $1/2$ we obtain a subsequential limiting distribution on finite semigroups, with positive mass at the trivial semigroup and unbounded support in size. This contrasts with random groups, which have order at most two with overwhelming probability at critical density.
\end{abstract}

\maketitle

\section{Introduction}

The theory of random groups, originating in the work of Gromov \cite{G93}, provides a framework for studying the properties of a ``typical'' finitely presented group. In Gromov's density model, one chooses a random set of relators of a prescribed length at a fixed density and studies the resulting group as this length tends to infinity. In this paper, we introduce an analogous density model for random semigroups and study which properties of random groups persist in this broader setting.

Let us define the model of our random semigroups. Fix $n\ge 2$, and consider the free semigroup $\F_n=\langle x_1,\dots,x_n\rangle$. For each $l\ge 1$, we denote by $B_{n,l}$ the set of words in $\F_n$ of length at most $l$. This set can be thought of as the ball of radius $l$ in the (undirected) Cayley graph of~$\F_n$, and has size $|B_{n,l}|=n+\cdots+n^l=n^l(\frac{n}{n-1}+o_l(1))$.

\begin{defn}
Fix positive integers $n\ge 2$ and $l\ge 1$. Given a density parameter $\d\in (0,1)$, we pick a uniformly random set of $\lfloor n^{2\delta l}\rfloor$ unordered pairs $\{u_i,v_i\}\subseteq B_{n,l},\ i\in I:=\{1,\dots,\lfloor n^{2\delta l} \rfloor\}$ and consider the quotient semigroup $M_{n,\d,l} \coloneqq \F_n / \langle u_i=v_i\ |\ i\in I\rangle$.
\end{defn}

We think of $M_{n,\d,l}$ as a generalized random variable $\Omega \rightarrow \{\text{Semigroups}\}$ where the probability space $\Omega$ is the collection of all sets of $\lfloor n^{2\delta l} \rfloor$ unordered pairs of elements from $B_{n,l}$.

We are interested in properties of the random semigroup which hold when the number of generators and density are fixed, while the length~$l$ tends to infinity.

\begin{defn}
Let $\mathcal{P}$ be a semigroup property. We say that a random semigroup with $n$ generators and density $\delta$ satisfies $\mathcal{P}$ \textbf{with overwhelming probability} if
\[
\lim_{l\rightarrow \infty} \Pr(M_{n,\d,l}\ \text{satisfies}\ \mathcal{P})=1.
\]
\end{defn}

In Gromov's density model, density $\d=\frac{1}{2}$ marks a phase transition: with overwhelming probability, the random group is either trivial or isomorphic to $\Z/2\Z$ when $\d>\frac{1}{2}$, and non-elementary hyperbolic when $\d<\frac{1}{2}$ \cite{G93,Ollivier}. Further model-theoretic evidence for the largeness of random groups when $\d<\frac{1}{2}$ is given in \cite{KMS}. If one samples freely reduced relators, the group is also trivial or isomorphic to $\Z/2\Z$ with overwhelming probability at the critical density $\d=\frac{1}{2}$ \cite{DJKLMS16}.

In this work we study the same phase transition for random semigroups. We show in Theorem~\ref{thm:big-density} that for density $\d>\frac{1}{2}$, the random semigroup $M_{n,\d,l}$ is trivial with overwhelming probability. Here we do not need to add an exceptional case as $\Z/2\Z$, which is the case since our relators do not have fixed length.

When $\d<\frac{1}{2}$, we show that $M_{n,\d,l}$ is typically large by showing that the following properties hold with overwhelming probability:
\begin{itemize}
    \item $M_{n,\d,l}$ contains nonabelian free semigroups (Proposition~\ref{prop:free-subsemigroups}), hence has exponential growth.
    \item For any field $\mathbb{K}$, the semigroup algebra $\mathbb{K}[M_{n,\d,l}]$ is a filtered Golod-Shafarevich algebra (Proposition~\ref{prop:GS}). It then follows that $M_{n,\d,l}$ has a periodic quotient of bounded index with exponential growth (Corollary~\ref{cor:periodic-quotient}).
    \item We study small cancellation properties of $M_{n,\d,l}$, which are defined in an analogous way to the standard small cancellation properties of groups. We show that if $\d<\frac{1}{2m}$, then $M_{n,\d,l}$ satisfies $C(m)$ with overwhelming probability (Proposition~\ref{prop delta C(m)}). This allows us to prove that when $\d<\frac{1}{6}$, the semigroup $M_{n,\d,l}$ is residually finite (Corollary~\ref{cor:M-res-fin}).
    \item For any given $N\ge 1$, the semigroup $M_{n,\d,l}$ has a finite homomorphic image of size $N$ (Proposition~\ref{prop:finite-hom-images}).
    \item $M_{n,\d,l}$ is neither left amenable nor right amenable (Corollary~\ref{cor:non-amenability}). In fact, we show that there is a (random) group $G_{n,\d,l}$ and a homomorphism $\pi\colon M_{n,\d,l}\to G_{n,\d,l}$ such that $\pi(M_{n,\d,l})$ generates $G_{n,\d,l}$ (as a group), and $G$ is non-elementary hyperbolic (Theorem~\ref{thm:group-hyperbolic}).
\end{itemize}

Finally, we study the critical case $\d=\frac{1}{2}$. We show that for $n\ge 3$, the semigroup $M_{n,1/2,l}$ is finite with overwhelming probability (Theorem~\ref{thm:critical-tight}). However, in contrast to the case of random groups, here we get nontrivial limiting behaviour: $M_{n,1/2,l}$ is trivial with positive limiting probability (Corollary~\ref{cor:critical-trivial}), but for every $N\ge 1$ there is a positive limiting probability that $M_{n,1/2,l}$ has size at least $N$ (Proposition~\ref{prop:critical-unbounded}).

\subsection*{Structure of the paper}

We begin with some preliminary results of the random semigroup and its random presentation in Section~\ref{sec:prelim}, covering the large density case. The small density case is divided into several parts. We cover largeness properties in Section~\ref{sec:small-density}, small cancellation in Section~\ref{sec:small-cancellation}, and hyperbolicity and non-amenability in Section~\ref{sec:amenability}. The final section (Section~\ref{sec:critical}) is devoted to study the critical case.

\subsection*{Notations}
We write $\F_n=\langle x_1,\dots,x_n\rangle$ for the $n$-generated free semigroup. We refer to elements of $\F_n$ as words. For a word $w\in\F_n$, we write $|w|$ for its length with respect to the standard generators of~$\F_n$, which is the homomorphism $|\cdot| \colon \F_n \rightarrow \mathbb{N}\cup \{0\}$ given by $|x_i|=1$ for all $1\leq i\leq n$.

\subsection*{Acknowledgements}

BG is supported by the Mary P. Dolciani Halloran Foundation, the Simons Foundation, and the PSC-CUNY Research Award Program.
GB is supported by the ERC consolidator grant CUTOFF (101123174). Views and opinions expressed are however those of the authors only and do not necessarily reflect those of the European Union or the European Research Council Executive Agency.

\section{Preliminary results}\label{sec:prelim}

For a set $X$, we denote by ${X \choose 2}$ the set of (unordered) pairs of elements from $X$. Given a set of relations $\mathcal{R} \subseteq {\F_n \choose 2}$, we obtain a semigroup presentation $\langle x_1,\dots,x_n\ |\ \mathcal{R} \rangle$. Given a relation $u=v$, we call $u,v$ `relation words'.

Let $B_{n,l}$ be the set of words in the $n$-generated free semigroup, whose length is at most $l$. Notice that $$|B_{n,l}|=\sum_{i=1}^{l} n^i = \frac{n^{l+1} - n}{n-1} = \left(\frac{n}{n-1}+o_l(1)\right)\cdot n^l$$
and
$$p_{n,l} := \bigg|{B_{n,l} \choose 2}\bigg| = \left( \frac{n^2}{2(n-1)^2} + o_l(1) \right) \cdot n^{2l}.$$

We recall the following fact.

\begin{fact} \label{fact}
Let $\pi\colon \F_n \twoheadrightarrow \langle x_1,\dots,x_n\ |\ \mathcal{R} \rangle$ be the natural projection. For two words $w,w'\in \langle x_1,\dots,x_n \rangle$ we have $\pi(w)=\pi(w')$ if and only if there is a sequence of words $w=w_0,w_1,\dots,w_m=w'$ in the free semigroup such that each $w_{i+1}$ is obtained from $w_i$ by replacing a subword $u$ of it by $v$, where $u=v$ is among the relation set $\mathcal{R}$. In particular, $\pi^{-1}(\pi(w))=\{w\}$ if and only if $w$ contains no relation word as a subword.
\end{fact}

\begin{prop}
\label{finite quotients prop} Suppose that $0<\delta<\frac{1}{2}$.
\begin{enumerate}
    \item Let $\theta>2\delta$. Let $\mathcal{R}_{n,\delta,l}$ denote a random set of $\lfloor n^{2\delta l} \rfloor$ defining relations.  Then, with overwhelming probability, all relation words are of length at least $(1-\theta)l$.
    \item Let $0 < \varepsilon < 1-2\delta$. For each $1\leq i\leq l$, let $a_i$ denote the number of relation words of length $i$ (counted with multiplicity\footnote{Meaning that if $u=v,u=w$ are both relations, and $|u|=i$, then they count twice toward $a_i$.}). Then $\Pr(a_i>n^{i-\varepsilon l}) \leq 2n^{(2\delta-1+\varepsilon)l}$ and $$\Pr(\exists 1\leq i\leq l:\ a_i > n^{i-\varepsilon l}) \leq  2l \cdot n^{(2\delta-1+\varepsilon)l}.$$
\end{enumerate}
\end{prop}

\begin{proof}
\begin{enumerate}
\item Recall that $p_{n,l}={n+\dots+n^l \choose 2}$ is the number of unordered pairs of elements from $B_{n,l}$. Denote by $N:=\lfloor n^{2\delta l} \rfloor = |\mathcal{R}_{n,\delta,l}|$.  Let $k = \lceil (1-\theta)l \rceil$.
The number of unordered pairs of words in $B_{n,l}$ both of which are of length at least $k$ is
\[
q_{n,k,l} :=  {n^k+n^{k+1}+\dots+n^l \choose 2}.
\]
We omit the subscripts of $p,q$ for brevity.
Denote by $A$ the event that all words in $\mathcal{R}_{n,\delta,l}$ are of length at least $k$.
Then
    \begin{align} \label{eq:1}
    \Pr(A) = \frac{{q \choose {N}}}
    { {p \choose {N}} } & =
    \frac{q(q-1)\dots (q-N+1)}{p(p-1)\dots (p-N+1)} \\ & \geq \left(\frac{q-N}{p-N}\right)^N = \left(1 - \frac{p-q}{p-N}\right)^N. \nonumber
    \end{align}

Now
\begin{align*}
    p-q & = {n+\dots+n^l \choose 2} - {n^k+\dots+n^l \choose 2} \\ & = {\frac{n^{l+1}-n}{n-1} \choose 2} - {\frac{n^{l+1}-n^k}{n-1} \choose 2} \\  & \leq \frac{(n^{l+1}-n)^2}{2(n-1)^2} - \frac{(n^{l+1}-n^k-n+1)^2}{2(n-1)^2} \\
    & \leq \frac{1}{2(n-1)^2} \left( (n^{l+1}-n)^2 - (n^{l+1} - 2n^k)^2 \right) \\
    & \leq \frac{1}{2(n-1)^2} \cdot 4n^{l+k+1} = \frac{2n}{(n-1)^2} \cdot n^{l+k}.
\end{align*}
Recall that, for some $c>0$ we have $p\geq c \cdot n^{2l}$, so $p-N\geq c \cdot n^{2l} - n^{2\delta l}$ and, since $\delta<1$, we can write $p-N\geq d\cdot n^{2l}$ for some $d>0$. By the above computation, $p-q \leq e\cdot n^{(2-\theta)l}$ for some $e>0$.
It follows from \eqref{eq:1} that (for $C=\frac{e}{d}$)
\begin{align*}
    \Pr(A) & \geq \left(1-\frac{e\cdot n^{(2-\theta)l}}{d\cdot n^{2l}}\right)^{N} \\ & \geq (1-C\cdot n^{-\theta l})^{n^{2\delta l}} \geq 1-C\cdot n^{(2\delta-\theta)l} \xrightarrow{l\rightarrow \infty} 1
\end{align*}
since $\theta>2\delta$.

\item

Writing $a_i$ as a sum of indicator variables,
\begin{align*}
\mathbb{E}[a_i] & =  \mathbb{E}\left[ \sum_{|w|=i} \sum_{v\in B_{n,l}\setminus \{w\}} \mathbb{I}_{\{w,v\} \in \mathcal{R}} \right] \\ & = n^i\cdot(|B_{n,l}|  - 1) \cdot \frac{N}{{B_{n,l} \choose 2}} \\ & = \frac{2Nn^i}{n+\dots+n^l} \\
&\leq 2n^{2\delta l + i-l}.
\end{align*}
By Markov's inequality,
\begin{align*}
    \Pr(a_i > n^{i-\varepsilon l}) & \leq \frac{\mathbb{E}[a_i]}{n^{i-\varepsilon l}} \\ & \leq \frac{2n^{(2\delta-1)l+i}}{n^{i-\varepsilon l}} = 2n^{(2\delta-1+\varepsilon) l},
\end{align*}
as claimed. By the union bound,
\[
    \Pr(\exists 1\leq i\leq l:\ a_i > n^{i-\varepsilon l}) \leq 2l \cdot n^{(2\delta-1+\varepsilon)l}.
\]
\end{enumerate}
\end{proof}

We next show that for density $\delta > \frac{1}{2}$, a random semigroup is trivial with overwhelming probability. This echoes the situation for Gromov random groups, which are either trivial or cyclic of order $2$ with overwhelming probability when $\delta>\frac{1}{2}$.

\begin{thm}\label{thm:big-density}
Let $n\geq 2.$ For every $\delta > \frac{1}{2}$,  $M_{n,\delta,l}$ is trivial with overwhelming probability.
\end{thm}

\begin{proof}
Consider the random graph $\mathcal{G}$ with vertex set $B_{n,l}$, where two vertices $u$ and $v$ are joined by an edge if the relation $u=v$ appears among the set $\mathcal{R}_{n,\delta,l}$ of defining relations of $M_{n,\delta,l}$. Then $\mathcal{G}$ is an Erd\H{o}s--R\'enyi graph with $k\coloneqq|B_{n,l}|$ vertices and $N\coloneqq \lfloor n^{2\delta l} \rfloor$ edges.

It is well known that an Erd\H{o}s--R\'enyi graph $\mathcal{G}(k,N)$ is connected with probability tending to $1$ as $k\rightarrow \infty$ whenever $N \gg k\log k$.
In our setting, $k \asymp n^l$, so this condition holds for all sufficiently large $l$ since
$\delta>\frac{1}{2}$.

If $\mathcal{G}$ is connected, then the elements of $B_{n,l}$ are all congruent modulo the relations $\mathcal{R}_{n,\delta,l}$. In particular, with overwhelming probability, we have $x_1=\dots=x_n$ and $x_1^2=x_1$ in $M_{n,\delta,l}$. Thus, $M_{n,\delta,l}$ is trivial with overwhelming probability.
\end{proof}

\section{Random semigroups with density $\delta<\frac{1}{2}$ are large}\label{sec:small-density}

\subsection{Free subsemigroups}

\begin{lem} \label{lem free subsemigroup}
    Let $M = \langle x_1,\dots,x_n \mid u_1=v_1,\dots,u_m=v_m \rangle$ be a finitely presented semigroup and let $\pi\colon \langle x_1,\dots, x_n \rangle \twoheadrightarrow M$ be the natural projection. Let $w_1, w_2 \in \langle x_1 \dots,x_n \rangle$ be distinct words in the free semigroup such that $$|w_1| = |w_2|\geq |u_1|,\dots,|u_m|,|v_1|,\dots,|v_m|.$$
    If none of $u_1,\dots,u_m,v_1,\dots,v_m$ is a subword of any of $w_1^2, w_1w_2, w_2w_1, w_2^2$, then $\pi(w_1),\pi(w_2)$ generate a free subsemigroup of $M$.
\end{lem}

\begin{proof}
Suppose that $\pi(w_{i_1}\dots w_{i_k})=\pi(w_{j_1} \dots w_{j_{k'}})$ where $i_l, j_{l'}\in \{1, 2\}$ for all $1\leq l \leq k$ and $1\leq l'\leq k'$. Notice that $w_{i_1}\dots w_{i_k}$ and $w_{j_1} \dots w_{j_{k'}}$ do not contain any relation word as a subword; indeed, since $|w_1|=|w_2|\geq |u_1|,\dots,|u_m|,|v_1|,\dots,|v_m|$, any occurrence of a relation word would have to fall within $w_{i_t}w_{i_{t+1}}$ or $w_{j_t}w_{j_{t+1}}$, which is impossible by assumption. Hence, by Fact \ref{fact},
\begin{align*} \{w_{i_1}\cdots w_{i_k}\} & = \pi^{-1}\left(\pi(w_{i_1}\cdots w_{i_k})\right)\\ & = \pi^{-1}\left(\pi(w_{j_1} \dots w_{j_{k'}})\right) = \{w_{j_1}\dots w_{j_{k'}}\}
\end{align*}
and since $|w_1|=|w_2|$ and $w_1\neq w_2$, we obtain that $k'=k$ and $i_1=j_1,\dots,i_k=j_k$. It follows that $\langle \pi(w_1),\pi(w_2)\rangle$ is a free subsemigroup of~$M$.
\end{proof}

\begin{prop}\label{prop:free-subsemigroups}
    Let $n\geq 2$. For every $\delta <\frac{1}{2}$, $M_{n,\delta,l}$ contains a nonabelian free subsemigroup with overwhelming probability.
\end{prop}

\begin{proof}
Let $N=\lfloor n^{2\delta l} \rfloor$ and let $M_{n,\delta, l} = \langle x_1,\dots,x_n \mid \mathcal{R}\rangle$ where $\mathcal{R}$ is $u_1=v_1,\dots,u_N=v_N$.

Denote $X=\{x_1,\dots,x_n\}$ and choose a pair of words $w_1, w_2\in X^l$ uniformly at random (with overwhelming probability, they are distinct).  Let $W_1 = w_1^2$, $W_2 = w_1w_2$, $W_3 = w_2w_1$, and $W_4 = w_2^2.$ For each one of them, the probability that it contains a given word $P$ of length $s$ ($s\leq l$) as subword is at most $2l\cdot n^{-s}$, since there are (at most) $2l$ possible starting positions, and the probability that a random sequence of $s$ letters coincides with $P$ is $n^{-s}$. Notice that since $|w_1|=|w_2|=l$ and $s\leq l$, the induced distribution on each length-$s$ window in each one of $W_1,\dots,W_4$ is uniform (each position of $w_1,w_2$ appears only once within that window). Denote by $\mathcal{E}$ the event that
\[ \exists \ 1\leq i\leq 4:\ W_i\ \substack{\text{contains a relation word} \nonumber \\ \text{as a subword.}} \]
We have
\begin{equation}\label{eq:2}
\Pr(\mathcal{E} \;|\; a_1,\dots,a_l) \leq 4 \cdot \sum_{i=1}^{l} a_i \cdot 2l \cdot n^{-i} = 8l\sum_{i=1}^{l} a_i n^{-i}
\end{equation}
where $a_i$ denotes the number of length-$i$ relation words in $\mathcal{R}$. Fix an arbitrary $0 < \varepsilon < 1-2\delta$. By Proposition \ref{finite quotients prop},
\[
    \Pr(\exists 1\leq i\leq l:\ a_i > n^{i-\varepsilon l}) \leq 2l \cdot n^{(2\delta-1+\varepsilon)l}
\]
and so by \eqref{eq:2},
\begin{align*}
    \Pr(\mathcal{E}) & \leq \Pr(\exists 1\leq i\leq l:\ a_i > n^{i-\varepsilon l}) + \Pr(\mathcal{E}\ \cap \ \forall 1\leq i\leq l:\ a_i \leq n^{i-\varepsilon l}) \\ & \leq 2l \cdot n^{(2\delta-1+\varepsilon)l} + 8l\sum_{i=1}^{l} n^{i-\varepsilon l}\cdot n^{-i} \\ & = 2l \cdot n^{(2\delta-1+\varepsilon)l} + 8l^2 \cdot n^{-\varepsilon l} \xrightarrow{l\rightarrow \infty} 0
\end{align*}
since $\varepsilon>0$ and $2\delta-1+\varepsilon<0$. Therefore, with overwhelming probability, there exist $w_1,w_2$ of length $l$ such that none of $w_1^2,w_1w_2,w_2 w_1,w_2^2$ contains a relation word as a subword, and by Lemma \ref{lem free subsemigroup}, the images of $w_1,w_2$ in $M_{n,\delta,l}$ generate a free subsemigroup.
\end{proof}

\begin{cor}
    Let $n\geq 2$. For every $\delta < \frac{1}{2},$ $M_{n,\delta,l}$ has exponential growth with overwhelming probability.
\end{cor}

\subsection{Random semigroups with $\delta<\frac{1}{2}$ are Golod--Shafarevich}

\begin{defn}
\label{def golod}
    Let $\mathbb{K}$ be a field and let $A = \mathbb{K}\langle x_1,\dots,x_n \rangle / \langle \mathcal{R} \rangle$ be a finitely generated $\mathbb{K}$-algebra. Given $f\in \mathbb{K}\langle x_1,\dots,x_n \rangle$ with zero constant term, let $\nu(f)$ denote the minimum length of a monomial on which $f$ is supported. Let $r_i := \{f\in \mathcal{R}\ |\ \nu(f)=i\}$. We say that $A$ is a (filtered) \textbf{Golod-Shafarevich algebra} if, for some $0 < \tau < 1$, the following is convergent and has a negative value
    \[
    1 - n \tau + \sum_{i=0}^{\infty} r_i \tau^i.
    \]
\end{defn}

Recall that Golod-Shafarevich algebras are of exponential growth and enjoy many additional `largeness' properties (see an example in Corollary~\ref{cor:periodic-quotient} below).

\begin{prop} \label{prop:GS}
Fix $n\geq 2$ and $0 < \delta < \frac{1}{2}$. Then $\mathbb{K}[M_{n,\delta,l}]$ is a filtered Golod-Shafarevich algebra with overwhelming probability.
\end{prop}

\begin{proof}
Let $M_{n,\delta,l}=\langle x_1,\dots,x_n \mid \mathcal{R} \rangle$ where $\mathcal{R} \subseteq {B_{n,l} \choose 2}$ is a random subset of size $N=\lfloor n^{2\delta l} \rfloor$. The semigroup algebra $\mathbb{K}[M_{n,\delta,l}]$ has a presentation $\mathbb{K}\langle x_1,\dots,x_n \rangle / \langle u-v \mid \{u,v\}\in \mathcal{R} \rangle$. Given a relation $u=v$ in $\mathcal{R}$, notice that $\nu(u-v)=\min\{|u|,|v|\}$. Let $r_i$ be the number of relations $u=v$ in $\mathcal{R}$ for which $\min\{|u|,|v|\}=i$.
Recall that in Proposition \ref{finite quotients prop}, we defined $a_i$ to be the number of relation words of length $i$ (counted with multiplicity). Hence, $r_i \leq a_i$. Fix any $ 0 < \varepsilon < 1-2\delta$.  By Proposition \ref{finite quotients prop},
\[
    \Pr(\exists 1\leq i\leq l:\ a_i > n^{i-\varepsilon l}) \leq 2l \cdot n^{(2\delta-1+\varepsilon)l}
\]
and so with overwhelming probability we may assume that $r_i \leq a_i\leq n^{i-\varepsilon l}$ for all $1\leq i\leq l$. In this case, fix some $0<\alpha<\varepsilon$ and take $\tau=n^{\alpha-1}$ and
\begin{align*}
1 - n \tau + \sum_{i=0}^{\infty} r_i \tau^i & \leq 1 - n\tau + \sum_{i=1}^{l} n^{i-\varepsilon l} \tau^i \\ & =
1 - n^{\alpha} + n^{-\varepsilon l}\sum_{i=1}^{l} n^{\alpha i} \\ & = 1 - n^\alpha + n^{-\varepsilon l} \cdot \frac{n^{\alpha(l+1)} - n^\alpha}{n^\alpha-1} \\ & \leq 1 - n^\alpha + n^{(\alpha-\varepsilon)l} \cdot \frac{n^\alpha}{n^\alpha - 1}
\end{align*}
which, when $l\gg 1$, is negative, as $\alpha-\varepsilon<0$. It follows that $\mathbb{K}[M_{n,\delta,l}]$ is Golod-Shafarevich with overwhelming probability.
\end{proof}

A sample evidence of the largeness of $M_{n,\delta,l}$ for $\delta<\frac{1}{2}$ implied by Proposition \ref{prop:GS} is the following (probably well-known) Burnside-type result (in the spirit of Wilson; see also Ershov):

\begin{cor}\label{cor:periodic-quotient}
    For every $n\geq 2$ and $\delta<\frac{1}{2}$, the semigroup $M_{n,\delta,l}$ has a periodic quotient of bounded index with exponential growth with overwhelming probability.
\end{cor}
\begin{proof}
Let $r_i$ denote the number of defining relations $f$ of $\mathbb{K}[M_{n,\delta,l}]$ with $\nu(f)=i$.
By Proposition \ref{prop:GS}, there is some $0<\tau<1$ such that $h := 1 - n \tau + \sum_{i=0}^{\infty} r_i \tau^i < 0$. Take $k\gg 1$ such that $\sum_{i=1}^{\infty} n^i \tau^{ki} = \frac{n\tau^k}{1-n \tau^k} < -h$. Then, adding to $\mathbb{K}[M_{n,\delta,l}]$ the relations of the form $f_w:=w^{k+1}-w^k$ for all $w$ word in $x_1,\dots,x_n$, we add $n^i$ relations $f_w$ with $\nu(f_w)=ki$, and
$$ 1 - n\tau +\sum_{i=0}^{\infty} r_i \tau^i + \sum_{i=1}^{\infty} n^i \tau^{ki} < 0,$$
hence $\mathbb{K}[M_{n,\delta,l}]/\langle w^{k+1}-w^k\ :\ w\in M_{n,\delta,l} \rangle$ is again a (filtered) Golod-Shafarevich algebra.
\end{proof}

\section{Small cancellation and Finite quotients}\label{sec:small-cancellation}

\subsection{Small Cancellation Properties}

Let $\langle x_1,\dots,x_n \rangle$ be a free semigroup and let $\mathcal{R}$ be a set of relations.

\begin{defn}
    A nonempty word $p$ is called a \textbf{piece} if it either appears as a subword in two distinct relation words, or within the same relation word in distinct positions.
\end{defn}

\begin{defn}
   Let $m\in \mathbb{N}$. We say that $\langle x_1,\dots,x_n \mid \mathcal{R} \rangle$ satisfies \textbf{condition $C(m)$} if no relation word can be expressed as a product of fewer than $m$ pieces.
\end{defn}

Every set of relations satisfies $C(1).$ A set of relations satisfies $C(2)$ if and only if no relation word is a proper subword of any other relation word. For a thorough account of conditions $C(m)$, see Remmers \cite{Remmers71}.

Kambites \cite{Kambites11} proved that in the few relation model, a random semigroup satisfies $C(m)$ with overwhelming probability, for all $m\in \mathbb{N}$. Thinking of the few relation model as a `density zero' model, we prove the following optimal behavior:

\begin{prop}
Let $n\geq 2$ and $m\geq 2$. For every $0 < \delta < \frac{1}{2m}$, $M_{n,\delta,l}$ satisfies $C(m)$ with overwhelming probability.
    \label{prop delta C(m)}
\end{prop}

\begin{proof}

Let $N=\lfloor n^{2\delta l} \rfloor$ and let $M_{n,\delta, l} = \langle x_1,\dots,x_n \mid \mathcal{R}\rangle$ where $\mathcal{R}$ is $u_1=u_{-1},\dots,u_N = u_{-N}$.

Recall $M_{n,\delta, l}$ satisfies $C(m)$ if, for every $1\leq q\leq m-1$ and every partition of a relation word $u = \alpha_1\cdots \alpha_{q}$ into $q$ non-empty subwords, at least one of the $\alpha_k$ is not a piece. If, for some $q<m-1$, $u=\alpha_1\cdots\alpha_q$ is a product of $q$ pieces, this can be expressed as a product of $m-1$ pieces since subwords of pieces are also pieces\footnote{To split a piece $\alpha_k$ into two subword pieces requires that $|\alpha_k| >1$, but this requirement may be ignored as $l\to \infty$}. Thus, it suffices to only consider partitions into exactly $m-1$ non-empty subwords.

In particular, $C(m)$ fails when there exists an $m$-tuple of relation words $u_{i_0}, \dots, u_{i_{m-1}}$, a partition $u_{i_0} = \alpha_1\cdots \alpha_{m-1}$, and positions $p_1, \dots, p_{m-1}$ such that every $k$, $\alpha_k$ occurs in $u_{i_k}$ at position $p_k$, excluding the natural occurrence of $\alpha_k$ in $u_{i_k}$ when $i_k = i_0.$ Here, we use the convention that a subword $\alpha$ occurs in a word $w$ at position $p$ if $w = \beta\alpha\gamma$ and $|\beta| = p$.

Let $S$ be the set of all ordered lists of the form $$(i_0, \dots, i_{m-1}, \alpha_1, \dots, \alpha_{m-1}, p_1, \dots, p_{m-1})$$
such that $i_k\in \{\pm 1, \dots, \pm N\}$ for $0\leq k\leq m-1$; $u_{i_0} = \alpha_1\cdots \alpha_{m-1}$ is a partition into $m-1$ non-empty subwords; and $0\leq p_k \leq |u_{i_k}| - |\alpha_k|$ for every $1\leq k\leq m-1.$ We exclude from $S$ any lists such that there exists $1\leq k\leq m-1$ for which $i_k = i_0$ and $p_k = |\alpha_1\cdots \alpha_{k-1}|$ (i.e., the $p_k$ is the naturally occurring position of $\alpha_k$ in $u_{i_0}$).

For $s\in S$, let $\mathbb{I}_s$ denote the indicator of the event that, for every $1\leq k\leq m-1$, $\alpha_k$ occurs in $u_{i_k}$ in position $p_k$. Let $$X  = \sum_{s\in S}\mathbb{\mathbb{I}}_s.$$ Then, $X\geq 1$ implies that $M_{n,\delta, l}$ fails $C(m).$ By Markov's inequality, $\Pr(X\geq 1) \leq \mathbb{E}(X)$. Thus, it suffices to show $$\mathbb{E}(X) = \sum_{s\in S} \mathbb{E}(\mathbb{I}_s)\xrightarrow{l\to\infty}0.$$

Fix $s = (i_0, \dots, i_{m-1}, \alpha_1, \dots, \alpha_{m-1}, p_1, \dots, p_{m-1})\in S.$ To calculate $\mathbb{E}(\mathbb{I}_s)$, we consider the number of generator constraints imposed by the subwords $\alpha_k$ sequentially. For each $k$, suppose that each $\alpha_j$ occurs in $u_{i_j}$ at position $p_j$ for every $j<k$, and consider the number of new generator constraints imposed by the prescription of $\alpha_k$ in $u_{i_k}$ at position $p_k$. A position of $\alpha_k$ does not impose a new generator constraint if it overlaps with a position of some $\alpha_j$ for $j<k$ or if it overlaps with the natural occurrence of $\alpha_k$ in $u_{i_0}.$ An overlapping position (of $\alpha_k$ with some $\alpha_j$, $j\leq k$) is called compatible if the prescribed generators match. For each $k$, denote the number of overlapping positions by $r_k$. Then, sequentially, $\alpha_k$ imposes $|\alpha_k| - r_k$ new generator constraints.

Note that the new generator constraints from $\alpha_1, \dots, \alpha_{m-1}$ are independent. Thus, letting $R(s) = \sum_{k=1}^{m-1}r_k$, $$\mathbb{E}(\mathbb{I}_s) = \prod_{k=1}^{m-1}n^{-(|\alpha_k|- r_k)} = n^{-|u_{i_0}| + R(s)}$$
provided that all overlapping constraints are compatible and 0 otherwise.

If $R(s) =0$, then $\mathbb{E}(\mathbb{I}_s) = n^{-|u_{i_0}|}.$ This, in particular occurs if $i_0, i_1, \dots, i_{m-1}$ are distinct. On the other hand, if $R(s)>0$, then $\mathbb{I}_s$ occurs with positive probability only if the partition $u_{i_0} = \alpha_1\cdots\alpha_{m-1}$ satisfies certain compatibility conditions (i.e., there must be sufficient overlap between or within the subwords $\alpha_k$).

We will use the following:

\begin{enumerate}
    \item Fix $2\delta(m-1) <\eps < 1 - 2\delta$ and let $a_L$ denote the number of length-$L$ relation words in $\mathcal{R}$. By \ref{finite quotients prop}, with overwhelming probability $$a_L\leq n^{L-\eps l} \,:\, 1\leq L\leq l.$$

    \item Fix $2\delta < \theta< \eps -2\delta(m-2)$. We claim that with overwhelming probability no relation word can satisfy more than $\theta l$ overlapping constraints.

    To see this, fix a relation word $u_i$ and an overlap pattern for a partition $u_i = \alpha_1\cdots \alpha_{m-1}$ in which there are at least $\theta l$ overlapping constraints. Here, the overlap pattern specifies identifications between positions within or between the subwords $\alpha_1, \dots, \alpha_{m-1}$. For a fixed such pattern with $r\geq\theta l$ constraints, the probability that $u_{i_0}$ can satisfy all overlapping constraints is $n^{-r}\leq n^{-\theta l}$.

    For a fixed relation word $u_i$, the number of possible partitions and overlap patterns is polynomial in $l$. Thus by the union bound, the probability that there exists a relation word satisfying an overlap pattern with more than $\theta l$ constraints is at most $$\text{poly}(l)Nn^{-\theta l} = \text{poly}(l)n^{(2\delta -\theta) l} \xrightarrow{l\to \infty}0$$ as $\theta>2\delta.$

\end{enumerate}

Thus, we will calculate $\mathbb{E}(X)$ under the assumptions that the conditions in (1) and (2) hold. In particular, under these assumptions, $$\mathbb{E}(\mathbb{I}_s) \leq \begin{cases}
    n^{-|u_{i_0}|} &\text{if } i_0, \dots, i_{m-1} \text{ pairwise distinct} \\
    n^{-|u_{i_0}| + \theta l} &\text{otherwise}.
\end{cases}$$

Let $1\leq L \leq l$. We now count the number of lists $s\in S$ such that $|u_{i_0}| = L$. There are $a_L$ choices for ${i_0}$. For the remaining $m-1$ indices ${i_1}, \dots, {i_{m-1}}$, there are at most $$(2N-1)(2N-2)\cdots(2N-m+1) \leq (2N)^{m-1}$$ choices such that $i_0, \dots, i_{m-1}$ are distinct; there are at most $${m\choose2}(2N)^{m-2}$$
choices where at least one index in $i_0, \dots, i_{m-1}$ is repeated. Note that these are upper bounds as we exclude lists for which there is at least one value $1\leq k\leq m-1$ such that $i_k = i_0$ and $p_k = |\alpha_1\cdots \alpha_{k-1}|$.

For a given $u_{i_0}$, there are $O(l^{m-2})$ choices for a partition $u_{i_0} = \alpha_1\cdots\alpha_{m-1}$, and $O(l^{m-1})$ choices for the positions $p_1, \dots, p_{m-1}$. Thus, \begin{align*}
    \mathbb{E}(X) &= \sum_{s\in S}\Pr(\mathbb{I}_s = 1) \\
    &=\sum_{\substack{s\in S \\
    \text{distinct $m$-tuple}}}\Pr(\mathbb{I}_s = 1)  + \sum_{\substack{s\in S \\
    \text{non-distinct $m$-tuple}}}\Pr(\mathbb{I}_s = 1) \\
    &\leq \sum_{L=1}^l a_L \text{poly}(l)(2N)^{m-1}n^{-L} + \sum_{L=1}^l a_L\text{poly}(l)(2N)^{m-2}n^{-L+\theta l}\ \\
    (*) &\leq\sum_{L=1}^l a_L \text{poly}(l)n^{2\delta(m-1)l}n^{-L} + \sum_{L=1}^l a_L\text{poly}(l)n^{2\delta(m-2)l}n^{-L+\theta l}.
\end{align*}

Under assumption (1), for all $1\leq L\leq l$, $a_L \leq n^{L-\eps l}$, so
\begin{align*}
    (*) &\leq \sum_{L=1}^l \text{poly}(l)n^{2\delta(m-1)l-L+L-\eps l} + \sum_{L=1}^l \text{poly}(l) n^{2\delta (m-2) l -L + \theta l + L- \eps l} \\
    &= \text{poly}(l)n^{(2\delta(m-1)-\eps )l} + \text{poly}(l)n^{(2\delta(m-2) + \theta -\eps)l} . \\
\end{align*}
Since $2\delta(m-1) < \eps$ and $2\delta(m-2)+\theta<\eps$, this tends to $0$ as $l\to \infty.$

\end{proof}

\subsection{Residual Finiteness}

\begin{defn}
A finitely generated semigroup $M$ is \textbf{graded} if there is a surjective homomorphism with finite fibers from a finitely generated free semigroup onto $M$.
\end{defn}

See \cite{MMS} for an analogous definition for monoids, where the following is proved for monoids $M$ in which $M \setminus \{1\}$ is a semigroup.

\begin{lem}[{\cite[Theorem~1.7]{MMS}}] \label{lem:res fin}
Every graded semigroup is residually finite.
\end{lem}

Further, Remmers \cite[Corollary~4.14]{Remmers71} proved that $C(3)$ semigroups are graded.

\begin{cor}\label{cor:M-res-fin}
For every $\delta<\frac{1}{6}$ and $n\geq 2$, $M_{n,\delta,l}$ is residually finite with overwhelming probability.
\end{cor}
\begin{proof}
This follows from Proposition \ref{prop delta C(m)}, Lemma \ref{lem:res fin}, and Remmers' theorem \cite{Remmers71}.
\end{proof}

\medskip

We close this section with the following remark, demonstrating the difference between the behavior of random semigroups and random groups. For random groups with $\d<\frac{1}{2}$, for every given $N$, a random group does not have a subgroup of index $\leq N$, or even a complex representation of dimension $\leq N$, with overwhelming probability (see~\cite{KL19} and~\cite{Chorna}).

\begin{prop}\label{prop:finite-hom-images}
Let $0 < \d < \frac{1}{2}$. For every $N\geq 2$, $M_{n,\d,l}$ has a finite homomorphic image of size $N$, with overwhelming probability.
\end{prop}

\begin{proof}
Let $0 < \delta < \frac{1}{2}$ be given. Fix $N\geq 2$. Let $t$ be such that $n+\cdots+n^{t-1} < N \leq n+\cdots+n^t$.

Fix some  $2\delta < \theta < 1$.
By Proposition \ref{finite quotients prop}, $$\Pr\left(\text{All relation words in }\mathcal{R}_{n,\delta,l}\ \text{are of length}\geq (1-\theta)l\right) \xrightarrow{l\rightarrow \infty} 1.$$
We may assume that $(1-\theta)l > t$. In the event that all words in $\mathcal{R}_{n,\delta,l}$ are of length $\geq (1-\theta)l$, the quotient semigroup $M_{n,\delta,l}$ maps onto
\begin{equation} \label{eq:3}
\widetilde{M} := \langle x_1,\dots,x_n\ |\ u=v\ \forall |u|,|v|\geq t+1 \rangle
\end{equation}
which is a semigroup with zero (represented by any word of length $\geq t+1$) whose cardinality is $n+\cdots+n^t$. Let $D:=n+\cdots+n^t-N \geq 0$ and pick $D$ arbitrary length-$t$ words over the alphabet $\{x_1,\dots,x_n\}$, say, $u_1,\dots,u_D$. Consider the quotient of the semigroup of $\widetilde{M}$ modulo the further relations $u_1=\cdots=u_D=0$. Now the non-zero words in this semigroup are precisely: all words of length $\leq t-1$; and all words of length $t$ except for $u_1,\dots,u_D$. Hence the cardinality of this quotient is $n+\dots+n^{t-1}+(n^t-D)=N$.
\end{proof}

\section{Hyperbolicity and non-amenability}\label{sec:amenability}

In this section, we will show that for every $n\ge 2$ and $0<\delta<1/2$, the random semigroup $M_{n,\delta,l}$ has a homomorphic image which generates a non-elementary hyperbolic group, and discuss applications to the (non) amenability of $M_{n,\delta,l}$.

\subsection{Hyperbolicity}
Given a semigroup $M$, let $M\xrightarrow{\iota} G(M)$ be its enveloping group: every homomorphism $f\colon M\rightarrow H$ where $H$ is a group factors uniquely through $\iota$. In particular, for any $n,\delta,l$, we define a random group $G_{n,\delta,l}=G(M_{n,\delta,l})$, namely, $G_{n,\delta,l}\coloneqq F_n/\llangle u_iv_i^{-1}\suchthat i\in I\rrangle$. We will prove that $G_{n,\d,l}$ is non-elementary hyperbolic with overwhelming probability.

Notice that this notion differs from Gromov random groups. Indeed, for any word $w$ in a free group, denote by $\cyc(w)$ its cyclically reduced form. Then the relations of $G_{n,\d,l}$ are sampled independently according to the following measure $\mu_l$: Take a uniform unordered pair $\{u,v\}\subseteq B_{n,l}$, choose an orientation uniformly, and then let $\mu_l$ be the measure of the random word $\cyc(uv^{-1})$.

Our goal will be to use a theorem of Ollivier, which we state below. Throughout the statement, ``a subword under $\mu_l$'' and ``two disjoint subwords under $\mu_l$'' are understood in the sense of Ollivier’s Definitions~6 and~7. In particular, the probability bounds are conditional on the total relator length and on specified portions revealed before the chosen subword. Furthermore, for two subwords, the definition separately covers two independent relators and two disjoint pieces of the same relator.

\begin{thm}[{\cite[Theorem 9]{Ollivier}}]\label{thm:ollivier}
Let $\Gamma$ be a non-elementary hyperbolic group with generating set $\{x_1^{\pm 1},\dots,x_{2n}^{\pm 1}\}$ and trivial virtual center. Let $\mu_l$ be a family of symmetric measures indexed by $l$, satisfying the following axioms:
\begin{enumerate}
    \item There is a constant $\kappa_1\ge 1$ such that $\mu_l$ weights only words of length between $l/\kappa_1$ and $\kappa_1 l$.

    \item There are constants $\kappa_2,\beta_2>0$ such that for every subword $x$ under $\mu_l$, for every $0\le t\le 1$, we have
    \[
    \Pr(\|x\|\le \kappa_2|x|(1-t)) \le e^{o(l)} (2n)^{-\beta_2 t|x|}.
    \]
    Here $\|\cdot\|$ stands for the word metric in the ambient group $\Gamma$.

    \item There are constants $\beta_3,\gamma_3>0$ such that for any two disjoint subwords under $\mu_l$, for any $m=m(l)$,
    \[
      \Pr(\exists \xi,\eta\in F_n\ :\ |\xi|,|\eta|\leq m,\ x\xi y\eta =1) \le e^{o(l)}\cdot (2n)^{\gamma_3 m}(2n)^{-\beta_3(|x|+|y|)}.
    \]

    \item There exist constants $\beta_4,\gamma_4>0$ such that, for any subword $x$ under $\mu_l$, for any $m=m(l)$,
    \[
      \Pr(\exists \xi,\eta\in F_n\setminus\{1\}\ :\ |\xi|,|\eta|\leq m,\ \xi x= x\eta) \le e^{o(l)}\cdot (2n)^{\gamma_4 m}(2n)^{-\beta_4|x|}.
    \]
\end{enumerate}
(We note that the equalities and inequalities of axioms (3) and (4) should hold in $\Gamma$, and that the $e^{o(l)}$ factors should be uniform in the conditioning data.)

Let $R$ be a set of independent random words of density at most $d$ picked under $\mu_l$, i.e., for every length $L$ there are at most $(2n)^{dL}$ (independent) words in~$R$ of length $L$ (each one sampled by $\mu_l^L=\mu_l(\cdot\;\big|\;|\cdot|=L)$).

If $d<\min\{\beta_2,\beta_3,\beta_4\}$, then with probability exponentially close to $1$ as $l\to\infty$, the random quotient $\Gamma/\llangle R\rrangle$ is non-elementary hyperbolic, as well as all intermediate quotients $\Gamma/\llangle R'\rrangle$ with $R'\subseteq R$.
\end{thm}

We want to apply Theorem~\ref{thm:ollivier} to the random group $G_{n,\delta,l}$. Here the ambient group is $\Gamma=F_n$, which is non-elementary hyperbolic and has trivial virtual center for $n\ge 2$. The measures $\mu_l$ are those defined above, which are symmetric by definition. However, these measures do not satisfy Axiom 1: they give a positive weight to short words. To overcome this, we will replace the measures $\mu_l$ by a truncated version of them. We fix $0\le\delta<\frac{1}{2}$ and choose $\delta<\theta<\frac{1}{2}$.

We write $f\ll_n g$ if there is a constant $C>0$, depending only on $n$, such that $f\le Cg$.

\begin{lem} \label{lem:cyc}
For $u,v$ chosen uniformly as above,
\[
\Pr(|\cyc(uv^{-1})| = L) \ll_n (l-\lceil L/2\rceil +1)n^{\lfloor L/2\rfloor -l}
\]
uniformly for $1\le L\le 2l$.
\end{lem}

\begin{proof}
To count the number of pairs $(u,v)$ with $|\cyc(uv^{-1})|=L$, we write $u=wpw'$ and $v=wqw'$ where $p,q$ do not begin with the same letter and do not end with the same letter. Then $|w|+|w'|+\max\{|p|,|q|\}\le l$, and $\cyc(uv^{-1})=pq^{-1}$, so $|\cyc(uv^{-1})|=|p|+|q|$. Therefore, the number of pairs $(u,v)$ such that $|\cyc(uv^{-1})|=L$ is bounded by the number of words $(w,w',p,q)$ of lengths $(a,a',i,L-i)$ such that $a+a'+\max\{i,L-i\}\le l$:
\begin{align*}
    \Pr(|\cyc(uv^{-1})|=L) &\le \frac{1}{|B_{n,l}|(|B_{n,l}|-1)}\sum_{i=\max\{0,L-l\}}^{\min\{L,l\}}\sum_{\substack{a+a'\le \\ l-\max\{i,L-i\}}}n^{a+a'+i+(L-i)} \\
    &\ll_n \frac{n^L}{n^{2l}}\sum_{i=\max\{0,L-l\}}^{\min\{L,l\}}\sum_{\substack{a+a'\le \\ l-\max\{i,L-i\}}}n^{a+a'}.
\end{align*}
We use the identity
\[
\sum_{a+a'\le M}n^{a+a'} = \sum_{j=0}^M (j+1)n^j \asymp_n (M+1)n^M
\]
and deduce
\begin{align*}
    \Pr(|\cyc(uv^{-1})|=L) &\ll_n \frac{n^L}{n^{2l}}\sum_{i=\max\{0,L-l\}}^{\min\{L,l\}}(l-\max\{i,L-i\}+1)n^{l-\max\{i,L-i\}} \\
    &\le \frac{l-\lceil L/2\rceil+1}{n^l}\sum_{i=\max\{0,L-l\}}^{\min\{L,l\}} n^{\min\{i,L-i\}} \\
    &\ll_n (l-\lceil L/2\rceil+1)n^{\lfloor L/2\rfloor -l}.
\end{align*}
\end{proof}

A simple union bound shows that:

\begin{cor}\label{cor:cyc-length-leL}
For any $L\ge 1$,
\[
\Pr(|\cyc(uv^{-1})| \le L) \ll_n l\,n^{L/2 -l}.
\]
\end{cor}

By taking another union bound, we get:

\begin{cor}\label{cor:exists-short-rel}
If $u_1=v_1,\dots,u_N=v_N$ are the relations of the random semigroup for $N=\lfloor n^{2\delta l}\rfloor$, then
\[
\Pr\left(\exists i:|\cyc(u_iv_i^{-1})|\le 2(1-2\theta)l\right) \ll_n l\, n^{2(\delta-\theta)l}.
\]
\end{cor}

We write $\widetilde{\mu}_l$ for the measure obtained from $\mu_l$ by conditioning on the fact that the length of a $\widetilde{\mu}_l$-random word is at least $cl$ for $c=2(1-2\theta)>0$. In other words, sampling according to $\widetilde{\mu}_l$ can be thought of as sampling according to $\mu_l$ until the resulting word is long enough.

We note that for any $cl\le L\le 2l$,
\[
\mu_l(\;\cdot\;\big|\;|\cdot|=L) = \widetilde{\mu}_l(\;\cdot\;\big|\;|\cdot|=L).
\]
Therefore, the verification of Axioms 3 and 4 for the truncated law $\widetilde{\mu}_l$ uses exactly the same length-conditioned distributions as $\mu_l$.

\begin{lem}\label{lem:conditional-uniform}
There is a constant $C=C_n$ such that the following holds. Let $r$ be a $\mu_l$-random word. Condition on $|r|=L$, fix a deterministic cyclic interval\footnote{Namely, a subset of the form $[i,j]$ or $[1,i]\cup[j,L]$ for $1\leq i\leq j\leq L$.} $I\subseteq\Z/L\Z$, and reveal all the letters outside $I$. Then, for every word $z$ whose letters are indexed by $I$,
\begin{equation}\label{eq:conditioned-mul}
  \Pr\left(r|_I=z\,\Big|\,|r|=L,\ r|_{I^c}\right)
\leq Cn^{-|I|}.
\end{equation}
\end{lem}

\begin{proof}
Write $r=\cyc(uv^{-1})$ for $u,v$ chosen uniformly from $B_{n,l}$. Write $u=wpw'$ and $v=wqw'$, where $w$ is the maximal common prefix of $u,v$, and $w'$ is the maximal common suffix of $u,v$ after removing $w$. Then $r=pq^{-1}$. Put $a=|p|$ and $b=|q|$. Conditional on $a,b$, all pairs $(p,q)$ with lengths $(a,b)$ such that $p,q$ do not start or end with the same letter, have the same probability: indeed, the number of possible choices of $w,w'$ is $\sum_{h+k\leq l-\max\{a,b\}}n^{h+k}$, which depends only on $a,b$. The cases $a=0$ or $b=0$ are the same, with fewer restrictions.

Take a cyclic interval $I\subseteq\Z/L\Z$ and a word $z$. Fix $a,b$ such that $a+b=L$, and fix a realization $\rho$ of $r|_{I^c}$ for which the conditioning event $|p|=a,|q|=b,r|_{I^c}=\rho$ has positive probability. Let $\mathcal{F}_{a,b}(\rho)$ be the set of all admissible fillings of the letters in $I$. Conditional on $a,b$ and on $r|_{I^c}=\rho$, all elements of $\mathcal{F}_{a,b}(\rho)$ have the same probability.

When $a,b>0$, the only restrictions on $p,q$ are $p_1\ne q_1$ and $p_a\ne q_b$. Since $\mathcal{F}_{a,b}(\rho)$ is nonempty, fix one admissible assignment to the at most four positions involved in the endpoint restrictions. All remaining positions in $I$ may be chosen freely, hence $|\mathcal{F}_{a,b}(\rho)|\ge n^{|I|-4}$. Therefore
\[
\Pr\left(r|_I=z \;\big|\; |p|=a,|q|=b,r|_{I^c}=\rho\right) \le n^4n^{-|I|}.
\]
The same bound holds also if $a=0$ or $b=0$. Since the right hand side does not depend on $a,b$, we may average over them and deduce
\[
\Pr\left(r|_I=z \;\big|\; |r|=L,r|_{I^c}=\rho\right) \le n^4n^{-|I|},
\]
so the lemma holds with $C=n^4$.
\end{proof}

The lemma conditions on the whole complement of the chosen interval, and is therefore stronger than each of the conditional estimates required in Ollivier’s theorem. Since the bound is uniform in the revealed letters, averaging over any information not prescribed by those definitions gives the required estimates.

\begin{prop}\label{prop:ollivier-axioms}
The measures $\widetilde{\mu}_l$ satisfy Ollivier's axioms with $\beta_2$ arbitrarily large, $\beta_3=\frac{1}{2}\log_{2n}n$, and $\beta_4=\log_{2n}n$.
\end{prop}

\begin{proof}
We do this directly by verifying each axiom:
\begin{enumerate}
    \item This is immediate by the definition of $\widetilde{\mu}_l$; indeed, we can take $\kappa_1 = \max\{2,c^{-1}\}$, since any word sampled by $\widetilde{\mu}_l$ has length between $cl$ and $2l$.

    \item In Ollivier's setting, the ambient group $\Gamma$ is allowed to be a non-elementary hyperbolic (not necessarily free) group. However, in our case $\Gamma$ is free. Since the words sampled by $\widetilde{\mu}_l$ are cyclically reduced, any subword $x$ of them is freely reduced, and so satisfies $\|x\|=|x|$. This proves Axiom 2 with $\kappa_2=1$ and arbitrarily large $\beta_2$.

    \item Let $x,y$ be two disjoint subwords and let $m=m(l)$. Fix words $\xi,\eta$ of lengths at most $m$. After revealing $x$, the equation $x\xi y\eta=e$ determines $y$ uniquely, if it has a solution. Hence, by \eqref{eq:conditioned-mul}, $\Pr(x\xi y\eta =e \big| \xi,\eta,x)\le Cn^{-|y|}$. If the equation holds, then $\bigl||x|-|y|\bigr|\le |\xi|+|\eta|\le 2m$, and therefore $|y|\geq\frac{|x|+|y|}{2}-m$. Summing over the at most $C'\cdot (2n)^{2m}$ choices of $\xi,\eta$ (for some $C'>0$), we obtain
    \begin{align*}\Pr\left(\exists\,\xi,\eta:\ |\xi|,|\eta|\leq m,\ x\xi y\eta =e\right) & \le
    C'\cdot (2n)^{2m} \cdot Cn^{-|y|}
    \\ & \le
    CC'(2n)^{3m}(2n)^{-\frac{1}{2}\log_{2n} (n)(|x|+|y|)}.
    \end{align*}
    We may therefore take $\beta_3=\frac{1}{2}\log_{2n}n$.

    \item Fix nontrivial words $\xi,\eta$ of length at most $m$. If $\xi x=x\eta$ has one solution $x_0$, then all solutions form the coset $x_0\text{Cent}_{F_n}(\eta)$. The centralizer of a nontrivial element of a free group is cyclic, and one may choose $|x_0|\le 2m$. If $\text{Cent}_{F_n}(\eta)=\langle h\rangle$, then any solution of length $s$ is of the form $x_0h^j$ for $|j|\le|h^j|\le|x_0|+|x_0h^j|\le 2m+s$. It follows that the number of reduced solutions $x$ of a prescribed length $s$ is at most $4m+2s+1$.

    Using \eqref{eq:conditioned-mul} and summing over $\xi,\eta$, we get
    \begin{align*}
        \Pr\left(\exists\,\xi,\eta\neq 1:\ |\xi|,|\eta|\leq m,\ \xi x=x\eta\right) & \le C(m+s)(2n)^{2m}n^{-s} \\ & \le e^{o(l)}(2n)^{3m}(2n)^{-\log_{2n}(n)s}.
    \end{align*}
    This proves the final axiom with $\beta_4=\log_{2n}n$.
\end{enumerate}
\end{proof}

We recall that a coupling of probability measures $\nu_1,\nu_2$ is a pair of random variables $(X_1, X_2)$, defined on a common probability space, with respective laws $\nu_1,\nu_2$. For every property $\mathcal{P}$,
\[
\left|\Pr(X_1\in \mathcal{P}) - \Pr(X_2\in \mathcal{P})\right| \le \Pr(X_1\ne X_2).
\]
Thus a coupling for which $\Pr(X_1\ne X_2)=o(1)$ allows one to transfer events holding with overwhelming probability from one model to the other.

\begin{thm}\label{thm:group-hyperbolic}
For any $n\ge 2$ and $0\le\delta<\frac{1}{2}$, the group $G_{n,\delta,l}$ is non-elementary hyperbolic with overwhelming probability.
\end{thm}

\begin{proof}
Fix $0\le\delta<\frac{1}{2}$ and $n\ge 2$. To apply Theorem~\ref{thm:ollivier}, we need to make sure that the measures $\mu_l$ satisfy the four axioms of the theorem, and to control the number of relations of $G_{n,\delta,l}$ of each length. A priori, this does not hold for $G_{n,\delta,l}$: the measure $\mu_l$ gives positive weight to short words as well as long ones, and the length of each relation is random. We fix this by using two couplings: The first one replaces $\mu_l$ by $\widetilde{\mu}_l$, which controls the lengths of the relations, and the second coupling controls the number of relations of each length, allowing us to get a good density for the set of relations.

For the proof, it will be easier for us to choose the relations of $G_{n,\delta,l}$ as independent words rather than choose an independent set of relations. Since $\delta<\frac{1}{2}$, this is harmless: If $Q=\binom{|B_{n,l}|}{2}$ is the total number of unordered pairs of words and $N=\lfloor n^{2\delta l}\rfloor$ is the number of sampled relations, the probability that we choose the same relation twice when taking $N$ independent pairs is bounded by $\frac{\binom{N}{2}}{Q} = O_n(n^{(4\delta-2)l})=o_l(1)$, so this event has probability tending to $0$ as $l$ tends to $\infty$.

For the first part, we let $\widetilde{G}_{n,\delta,l}$ denote a random group with $N=\lfloor n^{2\delta l}\rfloor$ relations, such that the relations are independent and sampled by $\widetilde{\mu}_l$. We claim that there is a coupling of $G_{n,\delta,l}$ and $\widetilde{G}_{n,\delta,l}$ such that $\Pr(G_{n,\delta,l}\ne \widetilde{G}_{n,\delta,l})=o_l(1)$. Therefore, $G_{n,\delta,l}$ is non-elementary hyperbolic if and only if $\widetilde{G}_{n,\delta,l}$ is.

Indeed, we define the coupling as follows. Sample $N$ independent $\mu_l$-words $W_1,\dots,W_N$. Then, for each $1\le i\le N$, define $\widetilde{W}_i$ as follows: If $|W_i|\ge cl$, we set $\widetilde{W}_i=W_i$, and otherwise sample $\widetilde{W}_i$ from $\widetilde{\mu}_l$. By Corollary~\ref{cor:exists-short-rel},
\begin{align*}
    \Pr(G_{n,\delta,l}\ne\widetilde{G}_{n,\delta,l}) &\le \Pr(\exists i: W_i\ne \widetilde{W}_i) \\
    &\le \Pr(\exists i:|W_i|\le cl) \ll_n ln^{2(\delta-\theta)l} = o_l(1).
\end{align*}

We are left with computing the density of the relations of $\widetilde{G}_{n,\delta,l}$. For each $cl\le L\le 2l$, let $M_{\delta,L}$ denote the number of relations of $\widetilde{G}_{n,\delta,l}$ of length $L$. Then
\begin{align*}
    \mathbb{E}[M_{\delta,L}] &= N\Pr_{w\sim\widetilde{\mu}_l}(|w|=L) \\ & = N\Pr_{w\sim\mu_l}(|w|=L \,|\, |w|\ge cl) \\
    &= N\frac{\Pr_{w\sim\mu_l}(|w|=L)}{\Pr_{w\sim\mu_l}(|w|\ge cl)} \\ & \ll_n n^{2\delta l}l n^{L/2-l} \\
    &\le l n^{(2\delta-1)l+L/2} \le ln^{\delta L}
\end{align*}
where the $\ll_n$ inequality uses Lemma~\ref{lem:cyc} and $\Pr_{w\sim\mu_l}(|w|\ge cl)=1-o(1)$, and the last inequality uses $L\le 2l$. This gives the required exponential upper bound on the expected number of relators of each length. However, the number of relators at each length is random, and therefore need not satisfy Ollivier’s
deterministic length-by-length cap on every realization.

We solve this problem by considering a different set of relations $\mathcal{R}$, which will contain the relations of $\widetilde{G}_{n,\delta,l}$ with overwhelming probability. Take $\eps>0$ such that $\delta+\eps<\frac{1}{2}$. For each $cl\le L\le 2l$, we let $N_L=\lfloor n^{(\delta+\eps)L}\rfloor$. Moreover, for each $cl\le L\le 2l$, let $W_{L,1},W_{L,2},\dots$ be i.i.d.\ $\widetilde{\mu}_l$-random words, where $W_{L,j}$ is conditioned to have length $L$. We sample $\widetilde{G}_{n,\delta,l}$ by first sampling for each $1\le i\le N$ the length $L_i$ of the word $\widetilde{W}_i$, according to the distribution of the lengths induced from $\widetilde{\mu}_l$. We then take $\widetilde{W}_i$ to be the first variable from $W_{L_i,1},W_{L_i,2},\dots$ which has not been used yet. This defines $\widetilde{G}_{n,\delta,l}$ with set of relations $\widetilde{R}_{n,\delta,l}$. We then define
\[
\mathcal{R} = \{W_{L,j}\suchthat cl\le L\le 2l,1\le j\le N_L\}.
\]
We note that, by Markov's inequality, for each $cl\le L\le 2l$ we have
\[
\Pr(M_{\delta,L} \ge N_L) \ll_n ln^{-\eps L} \le ln^{-\eps cl},
\]
so
\[
\Pr(\widetilde{R}_{n,\delta,l}\nsubseteq \mathcal{R}) \le \sum_{L=\lceil cl\rceil}^{2l}\Pr(M_{\delta,L} \ge N_L) \ll_n l^2n^{-\eps cl} = o_l(1).
\]

For each $L$, the set $\mathcal{R}$ contains at most $N_L\le n^{(\delta+\eps)L}$ relations of length $L$. Therefore $\mathcal{R}$ satisfies the conditions of Theorem~\ref{thm:ollivier} with density at most $(\delta+\eps)\log_{2n}n$ and $\min\{\beta_2,\beta_3,\beta_4\}=\frac{1}{2}\log_{2n}n$. Further, with probability $1-o_l(1)$, the set of relations $\widetilde{R}_{n,\delta,l}$ is a subset of $\mathcal{R}$. Ollivier’s theorem states simultaneously that all quotients obtained from subfamilies of $\mathcal{R}$ are non-elementary hyperbolic. Hence $\widetilde{G}_{n,\delta,l}$, and therefore $G_{n,\delta,l}$, is non-elementary hyperbolic with overwhelming probability.
\end{proof}

\subsection{Non-amenability}

We now discuss the (non)-amenability of the random semigroup $M_{n,\delta,l}$ for any $0<\delta<1/2$. Theorem~\ref{thm:group-hyperbolic} shows that the enveloping group $G_{n,\delta,l}$ is non-elementary hyperbolic with overwhelming probability, and is thus non-amenable. We use this to prove that $M_{n,\delta,l}$ is neither left amenable nor right amenable with overwhelming probability.

To do so, we recall the relevant notions of semigroup amenability. Let $M$ be a semigroup, and let $\ell^{\infty}(M)$ denote the space of bounded real-valued functions on $M$. A \textbf{mean} on $\ell^{\infty}(M)$ is a positive linear functional $\phi\colon\ell^{\infty}(M)\to\mathbb{R}$ such that $\phi(1)=1$, where $1$ denotes the constant function.

For every $a\in M$, define the left translation operator $L_a$ by $(L_af)(x)=f(ax)$. We say that $M$ is \textbf{left amenable} if there is a mean $\phi$ on $\ell^{\infty}(M)$ such that $\phi(L_af)=\phi(f)$ for every $a\in M$ and $f\in\ell^{\infty}(M)$. Right amenability is defined similarly, using the right translation operators $(R_af)(x)=f(xa)$.

\begin{lem}
Let $M$ be a semigroup, let $G$ be a group, and let $\pi\colon M\to G$ be a semigroup homomorphism such that $\pi(M)$ generates $G$ as a group. Suppose that $M$ is left or right amenable. Then $G$ is amenable.
\end{lem}

\begin{proof}
Suppose that $M$ is left amenable, and let $\phi$ denote a mean on $\ell^{\infty}(M)$ invariant under the left translation operators of $M$. Given $f\in\ell^{\infty}(G)$, define $\widetilde{\phi}\colon\ell^{\infty}(G)\to\mathbb{R}$ by $\widetilde{\phi}(f) = \phi(f\circ\pi)$. Clearly, $\widetilde{\phi}$ is a mean on $\ell^{\infty}(G)$. It is invariant under translation by elements of $\pi(M)$ by the choice of $\phi$. Next, for every $g\in\pi(M)$ and $f\in\ell^{\infty}(G)$, we have
\[
\widetilde{\phi}(L_{g^{-1}}f) = \widetilde{\phi}(L_gL_{g^{-1}}f) = \widetilde{\phi}(f).
\]
Thus $\widetilde{\phi}$ is also invariant under translation by
inverses of elements of $\pi(M)$. Since $\pi(M)$ generates $G$, it follows that $\widetilde{\phi}$ is invariant under translation by all elements of $G$. Therefore $G$ is amenable. The proof for right amenability is similar.
\end{proof}

Combining with the non-amenability of $G_{n,\delta,l}$, we conclude:

\begin{cor}\label{cor:non-amenability}
For every $n\ge 2$ and $0<\delta<1/2$, the semigroup $M_{n,\delta,l}$ is neither left amenable nor right amenable with overwhelming probability.
\end{cor}

We also recall the strong F\o{}lner conditions. A semigroup $M$ satisfies the \textbf{left strong F\o{}lner condition} if, for every finite subset $H\subseteq M$ and every $\eps>0$, there is a finite nonempty subset $F\subseteq M$ such that $|F\setminus aF|\le\eps|F|$ for everty $a\in H$. The right strong F\o{}lner condition is defined similarly, using $Fa$ in place of $aF$.

The left (respectively, right) strong F\o{}lner condition implies left (respectively, right) amenability (see for instance \cite{GK17}). Thus, by the preceding corollary, for every $n\ge 2$ and $0<\delta<1/2$, the random semigroup $M_{n,\delta,l}$ satisfies neither the left nor the right strong F\o{}lner condition with overwhelming probability.

\section{The critical case $\delta=\frac{1}{2}$}\label{sec:critical}

We now turn to the critical density case, i.e., $\delta=\frac{1}{2}$. We recall that the Erdos--Renyi graph $\mathcal{G}_{n,1/2,l}$ associated to $M_{n,1/2,l}$ is the graph with vertex set $B_{n,l}$, where we join two vertices $u_1$ and $v_1$ if $u_1=v_1$ appears within the defining relations of $M_{n,1/2,l}$. We begin with the following facts:

\begin{fact}
$\mathcal{G}_{n,1/2,l}$ is a random graph on $|B_{n,l}|=n^l(\frac{n}{n-1}+o_l(1))$ vertices, where we place $n^l$ edges uniformly at random among all sets of $n^l$ edges on $B_{n,l}$. The density parameter of $\mathcal{G}_{n,1/2,l}$ (i.e., the ratio of the number of edges to the number of vertices) is $\frac{n-1}{n}+o_l(1)$.

If $n\ge 3$, we have $\lim_{l\to\infty}(\frac{n-1}{n}+o_l(1))>\frac{1}{2}$, so there is a giant component in $\mathcal{G}_{n,1/2,l}$: There is some constant $0<c_n<1$ such that $\mathcal{G}_{n,1/2,l}$ contains a connected component with at least $c_n|B_{n,l}|$ vertices with high probability as $l$ tends to $\infty$. In fact, all other connected components will be substantially smaller, though we do not need this.

We will denote the giant component $\mathcal{C}_{n,1/2,l}$ (and, when it does not exist, take $\mathcal{C}_{n,1/2,l}$ to be the largest connected component of $\mathcal{G}_{n,1/2,l}$). Writing $H_l=\{|\mathcal{C}_{n,1/2,l}| \ge c_n|B_{n,l}|\}$, we have $P(H_l^c)=o_l(1)$. By symmetry, $\Pr(w\in\mathcal{C}_{n,1/2,l})\ge c_n\Pr(H_l)$ for every $w\in B_{n,l}$, and in fact
\begin{equation}\label{eq:giant-comp-pr}
\Pr(w_1,\dots,w_k\in\mathcal{C}_{n,1/2,l}) \ge \frac{\binom{\lfloor c_n|B_{n,l}|\rfloor}{k}}{\binom{|B_{n,l}|}{k}}\Pr(H_l)
\end{equation}
and
\begin{equation}\label{eq:giant-comp-avoid}
\Pr(\{w_1,\dots,w_k\notin\mathcal{C}_{n,1/2,l}\}\cap H_l) \le (1-c_n)^k
\end{equation}
for every distinct $w_1,\dots,w_k\in B_{n,l}$.

We remark that when $n=2$, there is no giant component in $\mathcal{G}_{2,1/2,l}$. Instead, the largest connected component has size of order $|B_{2,l}|^{2/3}$ in probability, so the probability that even two generators lie in the same connected component of $\mathcal{G}_{2,1/2,l}$ tends to $0$. However, it is reasonable that one can use the congruence properties to increase $\mathcal{G}_{2,1/2,l}$ and get a giant component with positive probability.
\end{fact}

Using this fact, we prove the following:

\begin{thm}\label{thm:critical-tight}
For every $n\ge 3$, the semigroup $M_{n,1/2,l}$ is nilpotent with overwhelming probability, and its size is asymptotically tight:
\[
\lim_{K\to\infty}\limsup_{l\to\infty} \Pr\left(|M_{n,1/2,l}|>K\right) = 0.
\]
\end{thm}

\begin{proof}
We first show that, with overwhelming probability, the element $z$ of $M_{n,1/2,l}$ corresponding to the giant component $\mathcal{C}_{n,1/2,l}$ is an absorbing zero. For each $1\le i\le n$, let $A_i$ denote the event that there exists a word $u$ of length $l-1$ such that $u,x_iu\in\mathcal{C}_{n,1/2,l}$. Let $X_i$ count such words~$u$. Conditional on $|\mathcal{C}_{n,1/2,l}|=s$, the vertex set of $\mathcal{C}_{n,1/2,l}$ is uniformly distributed among all subsets of $B_{n,l}$ of size $s$, and hence
\[
\mathbb{E}\left[X_i\mid |\mathcal{C}_{n,1/2,l}|=s\right] = n^{l-1}\frac{s(s-1)}{|B_{n,l}|(|B_{n,l}|-1)}.
\]
The pairs $\{u,x_iu\}$ are vertex-disjoint, so the indicators of $u\in\mathcal{C}_{n,1/2,l}$ and $x_iu\in\mathcal{C}_{n,1/2,l}$ have non-positive covariance under this conditioning. Therefore
\[
\operatorname{Var}\left(X_i\mid |\mathcal{C}_{n,1/2,l}|=s\right) \le \mathbb{E}\left[X_i\mid |\mathcal{C}_{n,1/2,l}|=s\right].
\]
For $s\ge c_n|B_{n,l}|$ and sufficiently large $l$, the conditional expectation is at least $\frac{c_n^2}{2}n^{l-1}$. By Chebyshev's inequality,
\[
\Pr(A_i^c\cap H_l)\le \frac{2}{c_n^2n^{l-1}}.
\]
Consequently,
\[
\Pr(A_1^c\cup\cdots\cup A_n^c) \le \Pr(H_l^c)+\frac{2n}{c_n^2n^{l-1}} \longrightarrow 0.
\]
We therefore have $z=x_iz$ for every $1\le i\le n$ with overwhelming probability. A similar reasoning shows that $z=zx_i$ for every $1\le i\le n$ with overwhelming probability, so $z$ is an absorbing zero in $M_{n,1/2,l}$. We denote this event by $Z_l$.

We next estimate the probability that every word of length $r$ contains a subword which lies in $\mathcal{C}_{n,1/2,l}$. Fix $0<\eps<1$, and let $r\le l$ be sufficiently large. For each word $w$ of length $r$, let $B_w$ denote the event that $w$ has a subword of length between $\frac{r}{2}$ and $\frac{3r}{4}$ which lies in $\mathcal{C}_{n,1/2,l}$.

Let $S_1$ denote the set of words of length $r$ that have at least $r^{1-\eps}$ distinct subwords of length $\lfloor\frac{r}{2}\rfloor$, and $S_2$ the words of length $r$ with fewer than $r^{1-\eps}$ such distinct subwords. If $w\in S_1$, then $w$ has at least $\frac{1}{2}r^{1-\eps}$ distinct subwords of length $L$ for any integer $\frac{r}{2}\le L\le\frac{3r}{4}$, since any distinct subword of length $\lfloor\frac{r}{2}\rfloor$ appears either as a prefix or as a suffix of such a word. Therefore, for sufficiently large $r$, the word $w$ has at least $\frac{1}{16}r^{2-\eps}$ distinct subwords of lengths between $\frac{r}{2}$ and $\frac{3r}{4}$.

By \eqref{eq:giant-comp-avoid}, for any $w\in S_1$,
\[
\Pr(B_w^c\cap H_l)\le (1-c_n)^{r^{2-\eps}/16}.
\]
If $w\in S_2$, then it has at least $\lfloor\frac{r}{4}\rfloor$ prefixes of lengths between $\frac{r}{2}$ and $\frac{3r}{4}$, so
\[
\Pr(B_w^c\cap H_l)\le (1-c_n)^{\lfloor r/4\rfloor}.
\]
However, by \cite[Theorem 4.2]{Kh26}, $|S_2|\le \exp(O_n(r^{1-\eps}))$. Together we get
\begin{align*}
  \Pr\left(\bigcup_{w:|w|=r}B_w^c\right) &\le \Pr(H_l^c) + |S_1|(1-c_n)^{r^{2-\eps}/16} +|S_2|(1-c_n)^{\lfloor r/4\rfloor}\\
  &\le \Pr(H_l^c) + n^r(1-c_n)^{r^{2-\eps}/16} + \exp(O_n(r^{1-\eps}))(1-c_n)^{\lfloor r/4\rfloor}.
\end{align*}
Taking $r=l$ in the preceding estimate and using $\Pr(Z_l)\to 1$, we see that, with overwhelming probability, every word of length $l$ equals $z$.

Combining this with $\Pr(Z_l)\to 1$, we see that, with overwhelming probability, every word of length $l$ equals $z$. Since $z$ is absorbing, every longer word also equals $z$, so $M_{n,1/2,l}$ is nilpotent with overwhelming probability.

More generally, on the event $Z_l\cap\bigcap_{w:|w|=r}B_w$, every word of length at least $r$ equals $z$. Therefore $|M_{n,1/2,l}|\le 1+n+\cdots+n^{r-1}$. Consequently, for every sufficiently large fixed $r$,
\begin{multline*}
    \limsup_{l\to\infty} \Pr\left(|M_{n,1/2,l}|>1+n+\cdots+n^{r-1}\right) \\ \le n^r(1-c_n)^{r^{2-\eps}/16} + \exp(O_n(r^{1-\eps}))(1-c_n)^{\lfloor r/4\rfloor}.
\end{multline*}
The right-hand side tends to $0$ as $r\to\infty$, proving that the size of $M_{n,1/2,l}$ is asymptotically tight.
\end{proof}

\begin{cor}
With overwhelming probability, the semigroup $M_{n,1/2,l}$ is finite.
\end{cor}

\begin{cor}\label{cor:critical-subsequential-lim}
Consider the sequence of distributions of $M_{n,1/2,l}$ on the set of finitely presented semigroups. Then this sequence has a subsequence of these measures that converges pointwise to a probability measure on finite nilpotent semigroups.
\end{cor}

The probability of obtaining the trivial semigroup remains bounded away from zero, as follows directly from the above fact:

\begin{cor}\label{cor:critical-trivial}
For every $n\ge 3$,
\[
\liminf_{l\to\infty}\Pr(M_{n,1/2,l}\text{ is trivial}) > 0.
\]
\end{cor}

\begin{proof}
By \eqref{eq:giant-comp-pr},
\[
\liminf_{l\to\infty}\Pr(x_1,\dots,x_n,x_1^2\in\mathcal{C}_{n,1/2,l}) \ge c_n^{n+1}.
\]
This proves the corollary since, if $x_1,\dots,x_n,x_1^2$ all lie in the same connected component of $\mathcal{G}_{n,1/2,l}$, then they are all equal in $M_{n,1/2,l}$ and so $M_{n,1/2,l}$ is trivial.
\end{proof}

On the other hand, arbitrarily large finite quotients occur with probabilities bounded away from zero:

\begin{prop}\label{prop:critical-unbounded}
For every $N\ge 1$,
\[
\liminf_{l\to\infty}\Pr(M_{n,1/2,l}\text{ has a homomorphic image of size }N) > 0.
\]
\end{prop}

\begin{proof}
Fix $N\ge 2$, and let $t\ge 2$ be such that $1+n+\cdots+n^{t-2} < N \le 1+n+\cdots+n^{t-1}$. We claim that
\begin{equation}\label{eq:len-rels-critical}
\liminf_{l\to\infty}\Pr(\text{all relation words of }M_{n,1/2,l}\text{ are of length at least }t) > 0.
\end{equation}
Once \eqref{eq:len-rels-critical} is established, the proposition will follow as in the proof of Proposition~\ref{prop:finite-hom-images}.

To prove \eqref{eq:len-rels-critical}, we follow the proof of Proposition~\ref{finite quotients prop}. The number of unordered pairs of elements from $B_{n,l}$ is
\[
p=p_{n,l}=\binom{n+\cdots+n^l}{2},
\]
whereas the number of unordered pairs of words in $B_{n,l}$ of length at least $t$ is
\[
q=q_{n,t,l}=\binom{n^t+\cdots+n^l}{2}.
\]
If $A_l$ is the event that all words in $\mathcal{R}_{n,1/2,l}$ have length at least $t$, then
\[
\Pr(A_l)  = \frac{\binom{q}{n^l}}{\binom{p}{n^l}} \ge \left(1-\frac{p-q}{p-n^l}\right)^{n^l}.
\]
We have
\[
p-q \le \frac{2n^{1+l+t}}{(n-1)^2} \le 8n^{l+t-1}
\]
for all $n\ge 2$, and $p-n^l\ge d n^{2l}$ for some constant $d>0$, hence
\[
\Pr(A_l)  \ge \left(1 - \frac{8n^{l+t-1}}{dn^{2l}}\right)^{n^l} = \left(1 - \frac{8n^{t-1}}{dn^{l}}\right)^{n^l} \xrightarrow[l\to\infty]{}\exp\left(-\frac{8}{d}n^{t-1}\right).
\]
This proves \eqref{eq:len-rels-critical}, and the proposition follows.
\end{proof}

\begin{ques}
Corollary~\ref{cor:critical-subsequential-lim} shows that, after passing to a subsequence, the measures of $M_{n,1/2,l}$ converge to a probability measure on finite nilpotent semigroups. Corollary~\ref{cor:critical-trivial} shows that this limiting measure gives positive mass to the trivial semigroup, while Proposition~\ref{prop:critical-unbounded} shows that it gives positive mass to semigroups of arbitrarily large sizes. It is therefore natural to ask:
\begin{enumerate}
    \item Do the measures of $M_{n,1/2,l}$ converge pointwise (or in any other sense), without moving to a subsequence?
    \item Which finite semigroups occur in the support of the limiting measure?
\end{enumerate}
\end{ques}

\end{document}